\documentclass[12pt]{article}
\usepackage[english]{babel}
\usepackage{newlfont}
\usepackage{amsbsy}
\usepackage[dvips]{graphicx}
\usepackage{color}
\usepackage{bbm}
\usepackage{graphicx, subfigure}
\usepackage{color}
\usepackage[center]{caption2}
\usepackage{amssymb}
\usepackage{amsmath}

\newcommand{\rk}{{\operatorname{rk}}}
\newcommand{\rkRange}{{\operatorname{rkRange}}}
\newcommand{\mrk}{{\operatorname{mrk }}}
\newcommand{\Mrk}{{\operatorname{Mrk }}}
\newcommand{\mi}{{\operatorname{mid }}}
\newcommand{\rad}{{\operatorname{rad }}}
\usepackage{latexsym}
\usepackage{amsthm}
\usepackage{amsthm}
\theoremstyle{plain}
\newtheorem{thm}{Theorem}
\newtheorem{cor}[thm]{Corollary}
\newtheorem{lem}[thm]{Lemma}

\newtheorem{defin}[thm]{Definition}
\theoremstyle{remark}
\newtheorem{rem}[thm]{Remark}

\newcommand{\R}{\mathbb{R}}
\newcommand{\Q}{\mathbb{Q}}

\newcommand{\N}{\mathbb{N}}

\newcommand{\I}{\mathbb{I}}
\newcommand{\al}{\pmb{\alpha}}

\usepackage{mathtools}
\def\alphaltiset#1#2{\ensuremath{\left(\kern-.2em\left(\genfrac{}{}{0pt}{}{#1}{#2}\right)\kern-.2em\right)}}
\usepackage[center]{caption2}

\begin{document}
\title{Generalization of real interval matrices to other fields}
\author{Elena Rubei}
\date{}
\maketitle

{\footnotesize\em Dipartimento di Matematica e Informatica ``U. Dini'', 
viale Morgagni 67/A,
50134  Firenze, Italia }

{\footnotesize\em
E-mail address: elena.rubei@unifi.it}

\def\thefootnote{}
\footnotetext{ \hspace*{-0.36cm}
{\bf 2010 Mathematical Subject Classification:} 15A99, 15A03, 65G40

{\bf Key words:} interval matrices, rank, rational realization}

\begin{abstract} 
An interval matrix is a  matrix 
whose entries are intervals in $\R$. We generalize this  concept, which has been broadly studied, to other fields. Precisely 
we define a rational interval matrix to be a matrix 
whose entries are intervals in $\Q$.
We prove that a (real) interval $p \times q$ matrix with the endpoints of all its entries in $\Q$ contains a rank-one matrix if and only if contains a rational rank-one matrix
and contains a matrix with rank smaller than $\min\{p,q\}$ if and only if it
 contains a rational matrix with rank smaller than $\min\{p,q\}$;
from these results and from the analogous criterions 
for (real) inerval matrices,
 we deduce immediately  a criterion to see when a rational interval matrix contains a rank-one matrix  and a criterion to see when it is full-rank, that is, all the matrices it contains are full-rank. Moreover, given a field $K$ and a matrix $\al$ whose entries are subsets of $K$, we describe  a criterion to find the maximal rank of a matrix contained in $\al$. 
\end{abstract}

\section{Introduction}

Let $p , q \in \N \setminus \{0\}$;
a    $p \times q$  interval matrix $\al$ is a  $p \times q$  matrix 
whose entries are intervals in $\R$; we usually denote the entry $i,j$, $\al_{i,j}$, by $[\underline{\alpha}_{i,j}, \overline{\alpha}_{i,j}]$ with $\underline{\alpha}_{i,j} \leq 
\overline{\alpha}_{i,j}$.
A  $p \times q$  matrix $A$  with entries in $\R$ is  said contained in  a $p \times q$ interval matrix $  \al $ if $a_{i,j} \in \al_{i,j} $ for any $i,j$. 
There is a wide literature about interval matrices 
and the rank of the matrices they contain.
In this paper we generalize the concept of 
interval matrix to other fields and we start the study of the range of the rank of the contained matrices. Before sketching our results, we illustrate shortly some of the literature 
on interval matrices and the rank of the contained matrices and we say also some words on  partial matrices and on the  matrices with a given sign pattern; these research fields are connected with the theory of interval matrices. 
 
  Two
of the most famous theorems on interval matrices are Rohn's theorems
on full-rank interval matrices.
We say that a  $p \times q $ interval matrix $\al$ has full rank if and only if  all the matrices contained in $\al$ have rank equal to $\min\{p,q\}$. 
For any  $p \times q $ interval matrix $\al = ( [\underline{\alpha}_{i,j}, \overline{\alpha}_{i,j}])_{i,j}$ with $\underline{\alpha}_{i,j} \leq \overline{\alpha}_{i,j}$,
let $\mi(\al)$, $\rad(\al)$ and $|\al|$ be respectively the midpoint, the radius and the modulus of $\al$, that is 
 the  $p \times q$   matrices such that $$ \mi (\al)_{i,j}= \frac{\underline{\alpha}_{i,j}+ \overline{\alpha}_{i,j}}{2}, \hspace*{1cm} \rad(\al)_{i,j}= \frac{\overline{\alpha}_{i,j}- \underline{\alpha}_{i,j}}{2}, \hspace*{1cm} 
|\al|_{i,j} = \max\{|\underline{\alpha}_{i,j}|,| \overline{\alpha}_{i,j}|\} $$ for any $i,j$. 
The following theorem
characterizes full-rank {\em square} interval matrices:

\begin{thm} {\bf (Rohn, \cite{Rohn})} \label{Rohn1}
Let   $\al =( [\underline{\alpha}_{i,j}, \overline{\alpha}_{i,j}])_{i,j}$ be  a  $p \times p$  interval matrix, where
 $\underline{\alpha}_{i,j} \leq \overline{\alpha}_{i,j}$  for any $i,j$.
Let $Y_p=\{-1,1\}^p$ and, for any $x
 \in Y_p$, denote by $T_x$ the diagonal matrix whose diagonal is $x$.
Then  $\al$ is a full-rank interval matrix 
if and only if,  for each
$x,y \in Y_p$,  $$\det\Big(\mi(\al)\Big) \,\det\Big(\mi(\al) - T_x \, \rad(\al) \, T_y\Big)>0. $$
\end{thm}

See  \cite{Rohn} and  \cite{Rohn2} for other characterizations. 
The following theorem characterizes 
full-rank $p \times q$ interval matrices, see
\cite{Rohn3}, \cite{Rohn4}, \cite{Shary}:

\begin{thm}  {\bf (Rohn)} 
A $p \times q$ interval matrix $\al$ with $p \geq q$ has full rank if and only if the system
of inequalities $$\hspace*{2cm} |\mi(\al) \, x| \leq \rad(\al) \,
|x|, \hspace*{1.5cm} x \in \R^q$$ has only the trivial solution $x=0$. 
\end{thm}

  A research area which can be connected with the theory of interval matrices 
  is the one of the partial matrices: let $K$ be a field; a partial matrix over $K$ is a 
matrix where only some of the entries are given and they are elements of 
$K$;
a completion of a partial  matrix is a specification of the unspecified entries. 
The problem of determining the maximal and the minimal rank of the completions of a partial matrix has been widely studied.
In particular, in \cite{CJRW}, Cohen,  Johnson,  Rodman and Woerdeman determined 
the maximal rank of the completions of a partial
matrix in terms of the ranks and the sizes of its maximal  
specified submatrices; see also \cite{CD}
for the proof.
 The problem of a theoretical characterization of   the minimal rank of   the completions of a partial matrix seems more difficult  and it has been solved only in some particular cases. We quote also the paper \cite{HHW}, where a criterion to say if a partial matrix has a completion of rank $1$ is established. 
  
  In \cite{Ru2} we generalized Theorem \ref{Rohn1} to general closed interval matrices, that is matrices whose entries are closed connected nonempty subsets of $\R$; obviously
  the notion of general closed interval matrices generalizes the one of partial matrices and the one of interval matrices.
  
Also for interval matrices, the problem of determining the miminal rank of the matrices contained in 
a given interval matrix seems much more difficult than the problem  of determining the 
maximal rank.
We recall that in \cite{Ru1}  we 
 determined the maximum rank of the matrices contained in a given interval matrix
and we gave  a theoretical
      characterization
 of interval matrices containing at least a  matrix of rank $1$. 
Precisely the last result is the following (where the word 
``reduced'' means  that every column and every row has at least one entry not containing $0$).

\begin{thm} \label{casopositivo}
 Let   $\al =( [\underline{\alpha}_{i,j}, \overline{\alpha}_{i,j}])_{i,j}$ be a   $p \times q $ reduced  interval matrix  with $p,q \geq 2$ and  $0  \leq \underline{\alpha}_{i,j} \leq \overline{\alpha}_{i,j}$
 for any  $i \in \{1,\dots,p\}$  and $ j \in \{1,\dots,q\}$. There exists 
 $A \in \al$ with $\rk(A)=1$ if and only if, 
 for any $h \in \N$
 with $2 \leq h \leq 2^{\min\{p,q\}-1}$, for any
  $i_1,\dots, i_h \in \{1,\dots, p\}$, for any
$j_1,\dots, j_h \in \{1,\dots, q\}$ and for any permutation  $\sigma $ on $h$ elements, we have:
\begin{equation} \label{assthm}
\underline{\alpha}_{i_1, j_1 }\dots \underline{\alpha}_{i_h, j_h} \leq \overline{\alpha}_{i_1, j_{\sigma(1)}}  \dots
 \overline{\alpha}_{i_h, j_{\sigma(h)}}.\end{equation}
\end{thm}

In the previous paper \cite{G-S}, the authors studied the complexity of an algorithm to decide if an interval matrix contains a rank-one matrix and proved that the problem is NP-complete.

Finally we quote another research area which can be related 
to partial matrices, to interval matrices and, more generally,
to general interval matrices: the one of the matrices with a given sign pattern; let $Q$ be a $p \times q$ matrix with entries in $\{+,-,0\}$; we say that $A \in M(p \times q, \R)$ has sign pattern $Q$ if, for any $i,j$, we have that  $a_{i,j} $ is 
positive (respectively negative, zero) if and only if $Q_{i,j}$ is 
$+$ (respectively $-,0$). Obviously the set of the matrices 
with a given sign pattern can be thought as a general interval matrices whose entries are from $\{ (0, +\infty), (-\infty,0) , [0]\}$.
There are several papers 
studying the minimal and maximal rank of the matrices with a given sign pattern, see for instance \cite{A1}, \cite{A2}, \cite{A3},  \cite{Shi}. In particular, in \cite{A1} and \cite{A2}
the authors proved that the minimum rank of the real matrices with a given sign pattern is realizable by a rational matrix
in case this minumum is at most $2$ or at least $\min\{p,q\}-2$.

As we have already said,
in this paper we generalize the  concept of interval matrices to other fields.  
We define a {\bf rational interval matrix} to be a matrix 
whose entries are intervals in $\Q$; we prove that a (real) interval $p \times q$ matrix
with the endpoints of all its entries 
  in $\Q$  contains a rank-one matrix if and only if contains a rational rank-one matrix
and contains a matrix with rank smaller than $\min\{p,q\}$ if and only if it
 contains a rational matrix with rank smaller than $\min\{p,q\}$;
from these results and from Theorem \ref{Rohn1} and Theorem \ref{casopositivo}
 we deduce immediately  a criterion to see when a rational interval matrix contains a rank-one matrix  and a criterion to see when it is full-rank, that is, all the matrices it contains are full-rank, see Section 3. 
 Moreover, in Remark \ref{KoBer}, we observe that from the papers \cite{Ber} and \cite{Ko} we can deduce that
    it is not true that, for any $r$,
 if an interval matrix   with  the endpoints of all its entries  in $\Q$  contains a rank-$r$ real matrix, then it contains   a rank-$r$ rational matrix.
 Finally, given a field $K$, we define a {\bf subset matrix over $K$} to be a matrix 
 whose entries are nonempty subsets of $K$ and  we describe  a criterion to find the maximal rank of a matrix contained in a subset matrix (see Section 4). 
 
\section{Notation and first remarks}

$\bullet$  Let $\R_{>0}$ be the set $\{x \in \R | \; x >0\}$ and
let $\R_{\geq 0}$ be the set $\{x \in \R | \; x  \geq 0\}$; we define analogously $\R_{<0}$ and $\R_{ \leq 0}$. We denote by $\I$ the set $\R -\Q$. 

$\bullet $ Throughout the paper  let $p , q \in \N \setminus \{0\}$. 

$\bullet $ Let $\Sigma_p$ be the set of the permutations on $\{1,....,p\}$. For any $\sigma \in \Sigma_p$ we denote 
   the sign of the permutation $\sigma$ by $\epsilon (\sigma) $.

$\bullet$ For any ordered multiset $J=(j_1, \dots , j_r)$,  a {\bf multiset permutation}  $\sigma(J)$ of $J$  is an ordered arrangement of  the multiset $\{j_1, \dots , j_r\}$, where each element appears as often as it does in $J$.

$\bullet$ For any field $K$, let $M(p \times q, K) $ denote the set of the $p \times q $  matrices with entries in $K$. For any $A \in M(p \times q, K)$, let $\rk(A) $ denote the rank of $A$
and let $A^{(j)}$ be the $j$-th column of $A$.

$\bullet $
For any vector space $V$ over a field $K$  and any $v_1,\dots, v_k \in V$, let $\langle v_1,\dots, v_k\rangle $ be the 
span of $v_1,\dots,v_k$.

$\bullet $
 Let   $\al $ be a $p \times q $  subset matrix
 over a field $K$.  
 Given a matrix $A \in M(p \times q, K) $, we say that  $ A \in \al $ if and only if  $a_{i,j} \in \al_{i,j} $ for any $i,j$.
  We define
$$\mrk(\al) = \min\{\rk(A)   | \; A \in \al     \},$$
$$\Mrk(\al) = \max\{\rk(A) | \; A \in \al     \}.$$
We call them respectively {\bf minimal rank} and {\bf maximal rank} of $\al$.
Moreover, we define 
$$\rkRange (\al) =  \{ \rk (A)  | \; A \in \al \};$$ 
we call it the {\bf rank range} of $\al$.

We say that an  entry of $\al$ is  {\bf degenerate} if its cardinality  is $1$.

\begin{rem}
Let $\al$ be a subset matrix over a field $K$. Observe that  $$\rkRange(\al) =
  [\mrk(\al)  , \Mrk(\al)] \cap \N.$$
\end{rem}

The proof is identical to the one of the case of interval matrices in \cite{Ru1} (Remark 3).

We defer to some classical books on interval analysis, such as \cite{Moore}, \cite{Neu} and \cite{Moore2} for the definition of sum and multiplication of two intervals. In particular, 
for any interval $\alpha$ in $\R$ and any interval $\beta$ either in 
$ \R_{>0}$ or in $\R_{<0}$,
we define $\frac{\alpha}{\beta}$ to be the set $\left\{\frac{a}{b} | \; 
a \in \alpha, \; b \in \beta \right\}$.
Obviously we can give analogous definitions for intervals in $\Q$.

\begin{defin}
Let $\al $ be an interval matrix, respectively a rational interval matrix. We say that another interval matrix (respectively rational interval matrix) $\al'$ is obtained from $\al $ by  an  {\bf elementary row operation} if it is obtained from $\al$ 
by one of the following operations (where we are considering interval arithmetic):

I) interchanging two rows,

II) multiplying a row by a nonzero real number (respectively rational number),

III)  adding to a row the multiple of another row by a real number  (respectively rational number).

In an analogous way we may define  {\bf elementary column operations}.
\end{defin}

\begin{rem} \label{op}

$\bullet$ 
Obviously the operations of the  first two kinds give an equivalence relation, but if we consider also the third kind we do not get an equivalence relation.

$\bullet$ 
Let $\al$ and $\al'$ be two interval matrices,
respectively two rational interval matrices, such that 
$\al'$ is obtained from $\al$ by elementary row (or column) operations. Then, obviously, 
\begin{equation} \label{elop}
\rkRange(\al) \subseteq  \rkRange (\al').
\end{equation}
Moreover, if $\al'$ is obtained from $\al$ only by elementary row (or column) operations of kind I or II, we have  the equality in (\ref{elop}).
\end{rem}

\begin{rem} \label{mrk0} Let $\al $ be 
a (rational) interval matrix.
 If $\al'$ is the (rational) interval matrix obtained from 
 $\al$ by deleting the columns and the rows such that all
 their entries contain $0$, we have that  $\mrk(\al)= \mrk(\al')$.
Obviously the analogous statement for $\Mrk$ does not hold.
\end{rem}

\section{Some results on rational interval matrices}

\begin{thm} \label{Q-R1}
Let  $ p \geq q$ and let   
$\al =( [\underline{\alpha}_{i,j}, \overline{\alpha}_{i,j}])_{i,j}$ be a $p \times q$ interval matrix with 
$\underline{\alpha}_{i,j} \leq \overline{\alpha}_{i,j}$ and 
 $\underline{\alpha}_{i,j}, \overline{\alpha}_{i,j} \in \Q$ for any $i,j$. 
If there exists $ A \in \alpha$ with $rk(A) <q$, then 
there exists $ B \in \alpha \cap M(p \times q, \Q)$ with $rk(B) <q$.
\end{thm}

\begin{proof}
We can suppose that $A^{(q)} \in \langle A^{(1)}, \ldots  A^{(q-1)} \rangle $; let $A^{(q)}  $ be equal 
to $$ c_1 A^{(1)}+ \ldots+ c_{q-1}  A^{(q-1)}  $$ 
for some $c_1,\ldots , c_{q-1} \in \R \setminus \{0\}$.
Up to swapping rows and columns, we can also suppose $c_1, \ldots, c_r \in \Q$,
$ c_{r+1}, \ldots , c_{q-1} \in \I$, $ a_{1,q}, \ldots ,
a_{k,q} \in \I$ and  $ a_{k+1,q}, \ldots ,
a_{p,q} \in \Q$.

Finally we can easily suppose that $a_{i,j} \in \Q $
for any $i=1,\ldots , k$ and $j=1,\ldots, q-1$; in
 fact:
 for any $i \in \{1, \ldots, k\}$, let $Z(i)= 
 \{j \in \{ 1, \ldots, q-1 \}| \;  a_{i,j}  \in \I\} $; if 
 for some $i \in \{1, \ldots , k \}$ the set $Z(i)$ is nonempty, we have that for any $j \in Z(i) $ the entry $\al_{i,j}$ is nondegenerate
 (since has rational endpoints and contains $a_{i,j}$, which is irrational); so there exist  neighbourhoods $U_{i,j}$ of $a_{i,j}$ 
 contained in $\al_{i,j}$  for  any $j \in Z(i) $ 
 such that 
$$ \sum_{j \in \{1, \ldots , q-1\} \setminus Z(i)} c_j a_{i,j} + \sum_{j \in Z(i)} c_j U_{i,j} \subset  \al _{i,q}  $$
(observe  that, for $ i=1, \ldots, k$, the entry 
$\al _{i,q}$ is nondegenerate, since it has rational endpoints and contains $a_{i,q}$ which is irrational); 
hence, for any $j \in Z(i)$ we can change the entry $a_{i,j} $ into an element $ \tilde{a}_{i,j}$  of $ U_{i,j} \cap \Q$ and the entry $a_{i,q}$ into
$$ \sum_{j \in \{1, \ldots , q-1\} \setminus Z(i)} c_j a_{i,j} + \sum_{j \in Z(i)} c_j \tilde{a}_{i,j} ;$$ in this way 
we get  again a matrix with the last column 
in the span of the first $q-1$ columns; moreover, no row of  the matrix we have obtained  has
all the entries but the last rational and 
the last entry irrational.  So  we can  suppose that $a_{i,j} \in \Q $
for any $i=1,\ldots , k$ and $j=1,\ldots, q-1$.

For any $ i=k+1,\ldots ,p$, we define:

$R^i =\{ j \in \{1,\ldots , r \} | \; a_{i,j} \in \Q \}$, 
 
$N^i =\{ j \in \{1,\ldots , r \} | \; a_{i,j} \in \I \}$, 
 
$\tilde{R}^i =\{ j \in \{r+1,\ldots , q-1\} | \; a_{i,j} \in \Q \}$,
 
$\tilde{N}^i =\{ j \in \{r+1,\ldots , q-1 \} | \; a_{i,j} \in \I \}$.

Moreover define  $$X=\{i \in \{k+1 , \ldots, p\} | \; 
N^i \cup \tilde{N}^i = \emptyset
\}.$$

 $\bullet$ Let 
  $i \in \{1,\ldots ,k\}$. For any $j=r+1,\ldots, q-1$, there exists a neighbourhood $V^i_j $ of $ c_j$ such that 
 \begin{equation} \label{star}
   \sum_{j=1, \ldots , r} c_j a_{i,j} + 
  \sum_{j=r+1, \ldots , q-1} V^i_j a_{i,j}  \subset    \al _{i,q}. 
  \end{equation}
 $\bullet $ Let $i \in \{k+1 , \ldots, p\} \setminus X$.
 By definition of the set $X$, there exists $ \overline{\jmath }(i) \in 
  N_i \cup \tilde{N}_i $. We consider
  neighbourhoods $V_j^i$ of $c_j$ contained either  in $\R_{<0}$ or in  $\R_{>0}$  
  for any $j  \in \{r+1, \ldots, q-1\} $ 
and neighbourhoods $U_{i,j}$ of $a_{i,j}$
contained in $\al_{i,j}$  for any 
$j  \in N_i \cup \tilde{N}_i \setminus  \{\overline{\jmath}(i)\}$  
such that 
\begin{equation} \label{star1}
-\frac{1}{c_{\overline{\jmath}(i)}} \left[ 
\sum_{j \in N_i \setminus \{\overline{\jmath}(i)\} } c_j U_{i,j}
+ \sum_{j \in \tilde{N}_i }
V^i_j U_{i,j} + \sum_{j \in \tilde{R}_i} V^i_j a_{i,j}
+ \sum_{j \in R_i} c_j a_{i,j} -a_{i,q} 
\right] \subset 
\al _{i, \overline{\jmath}(i)}
\end{equation} 
if $ c_{\overline{\jmath}(i)} \in \Q$   (i.e. $  \overline{\jmath}(i) \in \{1, \ldots r\}$) and
\begin{equation} \label{star2}
 -\frac{1}{V^i_{\overline{\jmath}(i)}} \left[ 
\sum_{j \in N_i  } c_j U_{i,j}
+ \sum_{j \in \tilde{N}_i \setminus \{\overline{\jmath}(i)\}}
V^i_j U_{i,j} + \sum_{j \in \tilde{R}_i} V^i_j a_{i,j}
+ \sum_{j \in R_i} c_j a_{i,j} -a_{i,q} 
\right] \subset 
\al _{i, \overline{\jmath}(i)}
\end{equation}
if $ c_{\overline{\jmath}(i)} \in \I$  (i.e. $  \overline{\jmath}(i) \in \{r+1, \ldots q-1\}$). 

\smallskip

  \underline{{\em Choice of the $\tilde{c}_j$ for $j=r+1, \ldots, q-1$.}}
  If $X = \emptyset$, for any $j=r+1, \ldots ,q-1$, choose $ \tilde{c}_j \in \left( \cap_{i \in \{1, \ldots, p\}} V^i_j  \right) \cap \Q$.
   
    If $X \neq \emptyset$,
   consider  the submatrix of $A$ given by 
 the rows indicized by $X$ and the columns $r+1, \ldots, q-1$ and  reduce it in  row echelon 
 form by elementary row operations; let $T$ 
 be the set of the $j \in \{r+1, \ldots, q-1\}$ corresponding to some pivot, and let $S$ be the set  $\{r+1, \ldots, q-1\} \setminus T$. For any 
 $j \in S$, choose $ \tilde{c}_j \in
 \left(  \cap_{i\in \{1, \ldots , p\} \setminus X}
V^i_j \right) \cap \Q$  in such a way that, called 
$\tilde{c}_j$ for $j \in T$ the solutions of the linear systems given by the equations 
\begin{equation} \label{sei}
\sum_{j \in R_i} c_j a_{i,j} + \sum_{j \in \tilde{R}_i}
\tilde{c}_j a_{i,j} = a_{i,q},
\end{equation}
  for $i \in X$,
we have that 
$ \tilde{c}_j \in  
\cap_{i\in \{1, \ldots , p\} \setminus X}
V^i_j $ for any $j \in T$. 

\smallskip

\underline{{\em Choice of the $\tilde{a}_{i,j} 
$ for $ i \in \{k+1,\ldots, p\} \setminus X $, $j \in N_i \cup \tilde{N}_i $.
}}
Now, 
for any $i \in \{k+1, \ldots , p\} \setminus X$, 
choose 
$\tilde{a}_{i,j} \in U_{i,j} \cap \Q$ for any $j \in N_i
\cup \tilde{N}_i \setminus \{\overline{\jmath}(i)\}$ and define $ \tilde{a}_{i, \overline{\jmath}(i) }$ to be 
\begin{equation} \label{sette}
 -\frac{1}{c_{\overline{\jmath}(i)}} \left[ 
\sum_{j \in N_i \setminus \{\overline{\jmath}(i)\} } c_j \tilde{a}_{i,j}
+ \sum_{j \in \tilde{N}_i}
\tilde{c}_j
 \tilde{a}_{i,j} + \sum_{j \in \tilde{R}_i} \tilde{c}_j a_{i,j}
+ \sum_{j \in R_i} c_j a_{i,j} -a_{i,q} 
\right] 
\end{equation}
if $ c_{\overline{\jmath}(i)} \in \Q$ (i.e. $\overline{\jmath}(i) \in N_i$),
\begin{equation} \label{otto}
-\frac{1}{\tilde{c}_{\overline{\jmath}(i)}} \left[ 
\sum_{j \in N_i  } c_j \tilde{a}_{i,j}
+ \sum_{j \in \tilde{N}_i \setminus \{\overline{\jmath}(i)\}}
\tilde{c}_j
 \tilde{a}_{i,j} + \sum_{j \in \tilde{R}_i} \tilde{c}_j a_{i,j}
+ \sum_{j \in R_i} c_j a_{i,j} -a_{i,q} 
\right] 
\end{equation}
if $ c_{\overline{\jmath}(i)} \in \I$
(i.e. $\overline{\jmath}(i) \in \tilde{N}_i$). By $( \ref{star1})$ and  $(\ref{star2})$, we have that $ \tilde{a}_{i, \overline{\jmath}(i) } \in \Q \cap \al_{i, \overline{\jmath}(i)}$.

Let $B$  be the $ p \times q $ matrix such that, 
 for every $ i=1,\ldots, p$  and $j=1, \ldots, q-1$,
 $$B_{i,j}= \left\{ \begin{array}{ll} 
 \tilde{a}_{i,j} &  \mbox{\rm if } a_{i,j} \in \I  \\ 
 a_{i,j} & \mbox{\rm if } a_{i,j} \in \Q \end{array}
 \right. $$ and such that 
$$B^{(q)}= 
\sum_{j=1, \ldots , r} c_j B^{(j)}+ \sum_{j=r+1, \ldots, q-1}
     \tilde{c}_{j} B^{(j)}.$$ 
     
     By the choice of $\tilde{c}_j$ for $ j=r+1, \ldots, q-1$ (see (\ref{sei})) and the choice 
     of $ \tilde{a}_{i,j}$ for $i \in \{k+1, \ldots, p\}
     \setminus X$, $j \in N_i \cup \tilde{N}_i$ 
     (see (\ref{sette}) and (\ref{otto})), we have 
     that $b_{i,q} =a_{i,q}$ for $i=k+1 , \ldots, p$. By the choice  of $\tilde{c}_j$ for $ j=r+1, \ldots, q-1$ such that 
$ \tilde{c}_j \in  
\cap_{i\in \{1, \ldots , p\} \setminus X}
V^i_j $  and by (\ref{star}), we get that $b_{i,q
} \in \al_{i,q}$ for $i=1,\ldots, k$.     
     So the matrix $B$ is contained in $ 
     \al \cap M(p \times q, \Q)$.
\end{proof}

From Theorem \ref{Rohn1} and Theorem
 \ref{Q-R1},
we get immediately the following corollary.

\begin{cor}
Let   $\al =( [\underline{\alpha}_{i,j}, \overline{\alpha}_{i,j}])_{i,j}$ be  a  $p \times p$  rational interval matrix, where
 $\underline{\alpha}_{i,j} \leq \overline{\alpha}_{i,j}$  for any $i,j$.
Let $Y_p=\{-1,1\}^p$ and, for any $x
 \in Y_p$, denote by $T_x$ the diagonal matrix whose diagonal is $x$.

Then  $\al$ is a full-rank rational interval matrix 
if and only if,  for each
$x,y \in Y_p$,  $$\det\Big(\mi(\al)\Big) \,\det\Big(\mi(\al) - T_x \, \rad(\al) \, T_y\Big)>0. $$
\end{cor}

Before stating the second theorem, we enunciate a lemma that will be useful in the proof of the theorem.

\begin{lem} \label{sys}
Let $A \in M(m \times n, \Q)$ for some $m,n \in \N  \setminus \{0\}$.
Let $c \in \R^n \setminus \{0\}$ be such that $Ac=0$.
Then, for every $V$ neighbourhood of $c$,
we can find $ \tilde{c} \in V \cap \Q^n$ such that $A \tilde{c}=0$. If, in addition,  $c \in (\R \setminus \{0\})^n$, we can find  $\tilde{c}$  
in $V \cap (\Q \setminus \{0\})^n$  such that $A \tilde{c}=0$.
\end{lem}

\begin{proof}
Let $\overline{A}$ be a matrix in row echelon form obtained from $A$ by elementary row operations. We can suppose
that the columns containing the pivots are the first $k$. Since $c$ is nonzero and $Ac=0$, we have that $k <n$.
Write $c= \left( 
\begin{array}{c} c'   \\ c'' \end{array}
\right)$ with $c'$ given by the first $k$ entries of $c$ and $c''$ given by the last $n-k$ entries and let $V'$ and $V''$ be respectively neighbourhood of $c'$ and $c''$ such that $V' \times V''$ is contained in $V$. 
There exists a neighbourhood $U$ of $c''$ 
contained in $V''$ such that, if $ b'' \in U$ and $\left( 
\begin{array}{c} b'  \\ b'' \end{array}
\right)$ is the solution of the linear system
$\overline{A} x =0$ with vector of the last $n-k$ entries equal to $ b''$, we have that 
$b' \in V'$.  
So if we take $ \tilde{c}'' \in U \cap \Q^{n-k}$
 and $\left( 
\begin{array}{c}  \tilde{c}'   \\ \tilde{c}'' \end{array}
\right)$ is the solution of the linear system
$\overline{A} x =0$ with vector of the last $n-k$ entries equal to $ \tilde{c}''$, we have that $\left( 
\begin{array}{c} \tilde{c}'   \\ \tilde{c}'' \end{array}
\right) \in V \cap \Q^n $. 

Finally, the last statement is obvious, because, if  $c \in (\R \setminus \{0\})^n$, we can find a neighbourhood $W$ of $c$ contained in  
$V \cap (\R \setminus \{0\})^n$ and, by applying the previous statement to $W$,  
 we get 
$\tilde{c} \in W \cap \Q^n
$  such that $A \tilde{c}=0$, thus 
$\tilde{c} \in V \cap (\Q \setminus \{0\})^n
$  and $A \tilde{c}=0$. 
\end{proof}

\begin{thm} \label{Q-R2}
Let  $ p \geq q$ and let   
$\al =( [\underline{\alpha}_{i,j}, \overline{\alpha}_{i,j}])_{i,j}$ be a $p \times q$ interval matrix with 
$\underline{\alpha}_{i,j} \leq  \overline{\alpha}_{i,j} $
and $\underline{\alpha}_{i,j}, \overline{\alpha}_{i,j} \in \Q$ for any $i,j$. 
If there exists $ A \in \alpha$ with $rk(A) =1$, then 
there exists $ B \in \alpha \cap M(p \times q, \Q)$ with $rk(B) =1$.
\end{thm}

\begin{proof}
We can suppose that every entry of $A^{(1)}$ is nonzero
and,  for $j=2, \ldots ,q$, we have that   $  A^{(j)} =c_j   A^{(1)}  $ for some $c_j \in \R
\setminus \{0\}$. For any $i \in \{1,\ldots , p\}$ such that 
$\al_{i,1}$ is nondegenerate, let 
$\tilde{\al}_{i,1}$ be a closed nondegenerate  interval neighbourhood of $a_{i,1}$, contained either in $
\al_{i,1} \cap \R_{>0}$ or  in $
\al_{i,1} \cap \R_{<0}$.

We can suppose also that 
$c_2, \ldots, c_k \in \I$ and 
$c_{k+1}, \ldots, c_q \in \Q$.

For any $ j=2, \ldots, q$, let $ I_j= \{i \in \{1, \ldots , p\} | \; a_{i,j} \in \I \}$ and let $ Q_j= \{i \in \{1, \ldots , p\} | \; a_{i,j} \in \Q \}$.

Let $j  \in\{ 2, \ldots , q\}$ and $ i\in I_j$; obviously  $\al_{i,j}$ is nondegenerate, because it has rational endpoints and contains $a_{i,j}$ which is irrational; 

if $a_{i,1} \in \I$, we define
$A_i^j $ to be a neighbourhood of $a_{i,1}$ 
contained either in $ \al_{i,1} \cap \R_{>0}$ or in  $ \al_{i,1} \cap \R_{<0}$, and, 
if $c_j \in \I$ (i.e. $ j \in \{2, \ldots,k\}$), we define
$V_i^j $ to be a neighbourhood of $c_j$, 
such that: 
$$
\begin{array}{ll}
 V_i^j a_{i,1}\subset \al_{i,j} & \mbox{\rm if }  \; a_{i,1} \in \Q
\;\; \mbox{\rm  and } \; c_j \in \I , \\
c_j A_i^j  \subset \al_{i,j}&   \mbox{\rm if }  \; a_{i,1} \in \I
\;\;  \mbox{\rm and } \; c_j \in \Q,  \\
 V_i^j A_i^j  \subset \al_{i,j} & \mbox{\rm if } \;  a_{i,1} \in \I
\;\;  \mbox{\rm and } \; c_j \in \I.
\end{array}
$$

For any $j=2, \ldots , k$, let $\tilde{c}_j $ be such that 
\medskip

(1) $\tilde{c}_j  \in \left( \cap_{i \in I_j} V^j_i \right) \cap  \left(   \cap_{i \in Q_j} \frac{ a_{i,j}}{\tilde{\al}_{i,1}} \right)  \cap (\Q \setminus \{0\})$

(observe that, if $i \in Q_j$, then, since $a_{i,j} \in \Q$ and $c_j \in \I$, we have that $a_{i,1}  \in \I$, so $\al_{i,1}$ is nondegenerate),

\medskip
(2) $ \frac{a_{i,j}}{ \tilde{c}_{j}}  = \frac{a_{i,j'}}{ \tilde{c}_{j'}} $ for any $i \in \{1,\ldots , p\}$ and $ j,j' \in \{2,\ldots, k \} $ such that $ i \in Q_j \cap Q_{j'}$,

\medskip
(3) $\tilde{c}_j \in \frac{a_{i,j}}{A_i^{j'}}$  for any $i \in \{1,\ldots , p\}$, $ j \in \{2,\ldots, k \}$,
$j' \in \{2,\ldots, q \} $ such that $ i \in Q_j \cap I_{j'}$.

By Lemma  \ref{sys}, we can 
find $\tilde{c}_j$ satisfying (1),(2),(3) because $V_i^j$ for $i \in I_j$,
$ \frac{ a_{i,j}}{\tilde{\al}_{i,1}}$ for $i \in Q_j$ 
and $\frac{a_{i,j}}{A_i^{j'}}$ for 
$ i \in Q_j \cap I_{j'}$ are neighbourhoods of 
$c_j$ and the equations in (2)
give a homogeneous linear system  in the variables $\tilde{c}_j $ satisfied by the $c_j$. 

\medskip
We define $B$ to be the matrix such that, for 
any $i=1, \ldots , p$, {\small
$$b_{i,1} = 
\left\{ 
\begin{array}{lll}
\frac{a_{i,j}}{ \tilde{c}_j}  & \mbox{\rm if } a_{i,1} \in \I \; \mbox{\rm and } i \in Q_j  \; \mbox{\rm for some  } j\in\{2, \ldots, k\}  , & (1^{st} \;  \mbox{\rm CASE})\\
 \mbox{\rm an element of  } \cap_{j=2,\ldots, q} A_i^j  \cap \Q & \mbox{\rm if } a_{i,1} \in \I \; \mbox{\rm and } i \in I_j \;  \forall j \in\{2, \ldots, q\} , 
  & (2^{nd} \;  \mbox{\rm  CASE})\\
a_{i,1} &  \mbox{\rm if  } a_{i,1}  \in \Q  & (3^{rd} \; \mbox{\rm  CASE})
\end{array}  
\right.
$$ }
and such that 
$$B^{(j)} = \left\{ 
\begin{array}{ll}
\tilde{c}_j B^{(1)}  &  \mbox{\rm for } j=2,\ldots , k ,\\
c_j B^{(1)}  &  \mbox{\rm for } j=k+1,\ldots , q. 
\end{array}
\right.$$

Observe that in the definition of $b_{i,1}$, the 
$1^{st}$ case and the $2^{nd}$ case 
cover all the case $ a_{i,1} \in \I$, because,
if $  a_{i,1} \in \I$ and $i \in Q_j$ for some 
$j \in \{2,\ldots, q\}$, then $ c_j \in \I$, so $ j \in \{2, \ldots, k\} $.

Observe also  that the definition of $b_{i,1}$ in the $1^{st}$ case is good by condition (2). Moreover  $ b_{i,1} \in \Q$ for any $i \in \{1, \ldots ,p \}$  and, 
finally,  $b_{i,1}$ is an element of $\al_{i,1}$: 
in the $1^{st}$ case, this follows from condition (1), in the other cases it is obvious.

Now we want to prove that 
$\tilde{c}_j B^{(1)} \in \al^{(j)}$ for $j=2, \ldots, k$ and that 
$c_j B^{(1)} \in \al^{(j)}$ for $j=k+1, \ldots, q$.

$\bullet$ 
First, let us prove that 
$\tilde{c}_j b_{i,1} \in \al_{i,j}$ for $j=2, \ldots, k$, $i =1, \ldots, p$. Let us fix $ i \in \{1, \ldots, p\}$.

$1^{st} $ CASE. In this case $a_{i,1} \in \I$,
$i \in Q_l$ for some $l \in \{2, \ldots, k\}$ and $b_{i,1}$ is defined to be $\frac{a_{i,l}}{ \tilde{c}_l} $; 
therefore, for $j=2, \ldots, k$, 
$$\tilde{c}_j b_{i,1} = \tilde{c}_j   \frac{a_{i,l}}{ \tilde{c}_l} = \, a_{i,j} \in \al_{i,j}$$
if $i \in Q_j$ (where the second equality holds by condition (2)), and
$$\tilde{c}_j b_{i,1} = \tilde{c}_j   \frac{a_{i,l}}{ \tilde{c}_l} \in  V_i^j A_i^j  \subset \al_{i,j}$$
if $i \in I_j$ (where the first  inclusion holds by conditions (1) and (3) and the second by the definition of $  V_i^j $ and $ A_i^j $).

$2^{nd}$ CASE. In this case $a_{i,1} \in \I$,
$i \in I_j$ for any $j \in \{2, \ldots,q\}$ and $b_{i,1}$ is defined to be a rational element of $\cap_{j=2,\ldots, q} A_i^j $; hence, for
$j=2, \ldots, k$, 
$$\tilde{c}_j b_{i,1}  \in  V_i^j A_i^j  \subset \al_{i,j}.$$

$3^{rd}$ CASE. In this case $a_{i,1} \in \Q$ and $b_{i,1}$ is defined to be $a_{i,1}$;  so, for any $j \in \{2, \ldots,k\}$, we get:
$$\tilde{c}_j b_{i,1}  = \tilde{c}_j a_{i,1}   
\in  V_i^j a_{i,1}  \subset \al_{i,j},$$
where the first inclusion holds by condition $(1)$ (observe that, since $c_j \in \I $ 
and $ a_{i,1} \in \Q$, we have that $a_{i,j} \in \I$, thus $i \in \I_j$) and the second by the definition of $V_i^j$.
 
$\bullet$ 
Finally, let us prove that 
$c_j b_{i,1} \in \al_{i,j}$ for $j=k+1, \ldots, q$,
$i=1, \ldots, p$.
Fix $i \in \{1, \ldots, p\}$.

$1^{st} $ CASE. In this case $a_{i,1} \in \I$,
$i \in Q_l$ for some $l \in \{2, \ldots, k\}$ and $b_{i,1}$ is defined to be $\frac{a_{i,l}}{ \tilde{c}_l} $; so, for $j=k+1, \ldots, q$, 
$$c_j b_{i,1} = c_j   \frac{a_{i,l}}{ \tilde{c}_l} \in
c_j  A_i^j \subset \al_{i,j},$$
 where the first inclusion  holds by condition (3) since $i \in Q_l \cap I_j$  and the last inclusion holds by definition
 of $A_i^j$.

$2^{nd}$ CASE. In this case $a_{i,1} \in \I$,
$i \in I_j$ for any $j \in \{2, \ldots, q\}$ and $b_{i,1}$ is defined to be a rational element of $\cap_{j=2,\ldots, q} A_i^j $;  hence, for $j=k+1, \ldots, q$, 
$$c_j b_{i,1}  \in c_j A_i^j  \subset \al_{i,j},$$
where the last inclusion holds by definition of 
$A_i^j$.

$3^{rd}$ CASE. In this case $a_{i,1} \in \Q$ and $b_{i,1}$ is defined to be $a_{i,1}$; 
so,  for $j=k+1, \ldots, q$,  we have:
$$c_j b_{i,1}  = c_j a_{i,1}   
=a_{i,j}  \in \al_{i,j}.$$

\end{proof}

Theorem \ref{casopositivo} and Theorem \ref{Q-R2} imply obviously  the following corollary.

\begin{cor} \label{casopositivoQ}
 Let   $\al =( [\underline{\alpha}_{i,j}, \overline{\alpha}_{i,j}])_{i,j}$ be a   $p \times q $ reduced rational interval matrix  with $p,q \geq 2$ and  $0  \leq \underline{\alpha}_{i,j} \leq \overline{\alpha}_{i,j}$
 for any  $i \in \{1,\dots,p\}$  and $ j \in \{1,\dots,q\}$. There exists 
 $A \in \al$ with $\rk(A)=1$ if and only if, 
 for any $h \in \N$
 with $2 \leq h \leq 2^{\min\{p,q\}-1}$, for any
  $i_1,\dots, i_h \in \{1,\dots, p\}$, for any
$j_1,\dots, j_h \in \{1,\dots, q\}$ and for any $\sigma \in \Sigma_h$, we have:
$$
\underline{\alpha}_{i_1, j_1 }\dots \underline{\alpha}_{i_h, j_h} \leq \overline{\alpha}_{i_1, j_{\sigma(1)}}  \dots
 \overline{\alpha}_{i_h, j_{\sigma(h)}}.$$
\end{cor}

Observe that,
as for (real) interval matrices (see Remarks 8 and 9 in \cite{Ru1}),
 to study when a reduced rational interval matrix contains a rank-one matrix it is sufficient to study the problem for a reduced rational interval matrix $ \al $,  
with $\al_{i,j} \subseteq \R_{\geq 0}$
 for every $i,j $.

In fact:
 let $\al$ be a $p \times q$ rational interval matrix.   Let $i \in \{1,\dots,p\}$ and $ j \in \{1,\dots,q\}$ be
 such that  $\underline{\alpha}_{i,j} \leq 0 
 \leq   \overline{\alpha}_{i,j}$.
Define  $ \al'$ and $ \al''$  to be the 
rational  interval matrices  such that $ \al'_{i,j}=[ \underline{\alpha}_{i,j},0]$,
 $  \al''_{i,j}=[0, \overline{\alpha}_{i,j}]$ and 
$\al'_{t,s}= \al''_{t,s} = \al_{t,s}$ for any $(t,s) \neq (i,j)$ (observe that obviously the definition of $\al'$ and $\al''$ do depend on $i,j$ we have fixed).
 Then 
 $$ \{A \in \al \} = \{A \in \al' \} \cup 
  \{A \in \al'' \};$$
 hence there exists $A \in \al $ with $\rk(A) =r$, for any $r \in \N$,
if and only if  either 
there exists $A \in \al' $ with $\rk(A) =r$ or 
there exists $A \in \al'' $ with $\rk(A) =r$.
In particular, to study whether a rational interval matrix
  $\al $
contains a rank-$r$ matrix, it is sufficient to consider the case where,  for any $i,j$, either $\al_{i,j} \subseteq \R_{\geq 0}$ or $\al_{i,j} \subseteq \R_{\leq 0}$. 
Observe that splitting every entry of $\al $ 
into the nonnegative part and the nonpositive part  can give $2^{p,q}$ matrices in the worst case.

 Moreover, by Remark
\ref{op}, we can suppose $\al_{i,j} \subseteq \R_{\geq 0}$ for every $(i,j) $ 
such that  either $i$ or $j$ is equal to $1$.
Finally for such a matrix $\al$,  
 if there exists $(i,j)$ such that $\al_{i,j} \subseteq
\R_{< 0} $, then $\al $ does not contain a rank-one matrix. 
Otherwise, that is $\overline{\alpha}_{i,j} \geq 0$ for any $i,j$,
 define $\hat{\al}$ to be the rational interval matrix  such that 
 $$\hat{\al}_{i,j} =
[ \max\{0, \underline{\alpha}_{i,j}\}, \overline{\alpha}_{i,j}]$$ 
 for any $i,j$. Obviously 
$\al $ contains a rank-one matrix if and only if 
$\hat{\al}$ contains a rank-one matrix.

\smallskip

\begin{rem} \label{KoBer}
In \cite{Ko} and \cite{Ber} the authors showed that there exists a sign pattern $Q$ such that the minimal rank $r^{\R}_{Q}$  of the real matrices with sign pattern $Q$ is strictly smaller 
than  the minimal rank $r^{\Q}_{Q}$ of the rational matrices with sign pattern $Q$. Let $A $ be a real matrix  with sign pattern $Q$ 
and rank $r^{\R}_{Q}$. Let $\al$ be an interval matrix containing $A$ and such that, for any $i,j$, we have:

$\al_{i,j}=\{0\}$ if and only if $a_{i,j}=0$,

$\al_{i,j} \subset \R_{>0}$ if and only if $a_{i,j}>0$,

$\al_{i,j} \subset \R_{<0}$ if and only if $a_{i,j}<0$. 

Obviously
$\mrk(\al)= r^{\R}_{Q}$ and, since there does not exist a rational matrix with sign pattern $Q$ and rank  $r^{\R}_{Q}$,
there does not exist a rational  matrix in $\al $ with rank $r^{\R}_{Q}$.
 So Theorem \ref{Q-R1} and Theorem \ref{Q-R2}  are not generalizable to any rank, that is, it is not true for any $r$,
 that, 
 if an interval matrix contains a rank-$r$ real matrix, then it contains   a rank-$r$ rational matrix.

\end{rem}

\section{Maximal rank of matrices contained in a subset matrix over any field}

\begin{defin}
Given a $p \times p $ subset  matrix over e a field $K$, $\al$, a {\bf partial generalized
diagonal} ({\bf pg-diagonal} for short) of length $k$ of $\al$ is a $k$-uple of the kind $$
(\al_{i_1, j_1},\dots, \al_{i_k, j_k})$$ 
for some  $\{i_1, \dots i_k\}$ and $ \{j_1, \dots, j_k\} $ subsets of $ \{1,\dots ,p\}$.

Its {\bf complementary matrix} is defined to be the submatrix of $\al$ given by the rows and columns whose indices are respectively in  $\{1,\ldots , p\} \setminus \{i_1, \ldots , i_k\}$ and 
in  $\{1,\ldots , p\} \setminus \{j_1, \ldots , j_k\}$. 

We say that a pg-diagonal is {\bf totally nondegenerate} if and only if all its entries are not degenerate.

We define $\det^c(\al) $ to be
$$ \sum_{\sigma \in \Sigma_p \; s.t. \; \al_{1, \sigma(1)}, \dots, \al_{p, \sigma(p)} \;  are \; degenerate} \epsilon (\sigma) \,
\al_{1, \sigma(1)} \cdot \ldots  \cdot\al_{p, \sigma(p)} $$
 if  there exists $\sigma \in \Sigma_p $ such that $  \al_{1, \sigma(1)}, \dots, \al_{p, \sigma(p)}$ 
  are  degenerate; we define $\det^c(\al) $ to be equal to $0$ otherwise.

For every pg-diagonal of length $p$, say 
$
\al_{1, \sigma(1)} \cdot \ldots  \cdot\al_{p, \sigma(p)} $
for some $\sigma \in \Sigma_p$, 
 we call  $\epsilon (\sigma) $ also  the sign of the pg-diagonal.
\end{defin}

In \cite{H} 
Hlad\'ik introduced the notion of strongly singular interval matrix. We generalize it to 
subset matrices over any field.

\begin{defin} 
(Hlad\'ik).
Let $\al $ be a $p \times p $ subset matrix over a field $K$. We say that it is strongly singular if $\Mrk(\al) < p$, that is, if every $A
\in \al $ is singular. 
\end{defin}

\begin{thm} \label{Mrkquadrate}
Let $\al $ be a $p \times p $ subset matrix over a field $K$. Then $\al $ is {\bf strongly singular} if and only if the following conditions hold:

(1) in $\al$ there is no totally nondegenerate pg-diagonal of length $p$,

(2) the complementary matrix of every  totally nondegenerate pg-diagonal of  length between $0$ and $p-1$ has $\det^c $ equal to $0$ (in particular $\det^c (\al)=0$).

\end{thm}

The proof  is quite similar to the one of Theorem 13 in \cite{Ru1}; for the convenience of the reader we sketch the proof here.

\begin{proof}
$\Longrightarrow$ 
We argue by induction on $p$. For $p=1 $ the statement is obvious. 
 Suppose $p \geq 2$ and that  the statement is true for $(p-1) \times (p-1) $ subset matrices. Let $\al$ be a $p\times p $ subset matrix such that 
 $\Mrk(\al) <p$; so  $\det(A)=0$ for every 
 $A \in \al$. 
 
If $\al$ contained a  totally nondegenerate  pg-diagonal
of length $p$, say $\al_{i_1,j_1}, \ldots
, \al_{i_p, j_p}$, then  $\al_{\hat{i_1}, \hat{j_1}}$ would have obviously a totally nondegenerate pg-diagonal of length $p-1$; hence, by induction assumption, there would  exist $B \in 
\al_{\hat{i_1}, \hat{j_1}} $ with $\det(B) \neq 0$. Hence, for any choice of elements $x_{i_1, j}\in \al_{i_1, j}$  for $j \neq j_1$ and
 $x_{i, j_1} \in \al_{i, j_1}$  for $i \neq i_1$, 
we could find $x \in \al_{i_1,j_1}$ such that the determinant of the matrix $X$ defined by $X_{\hat{i_1}, \hat{j_1}}=B$, $X_{i_1,j_1}=x$, $X_{i,j_1} = x_{i, j_1}$
for any $i \neq i_1$ and $X_{i_1,j} = x_{i_1, j}$
for any $j \neq j_1$  is nonzero, which is absurd. 
So (1) holds. 

Now, by contradiction, 
suppose (2) does not hold. 
Thus in $\al$ there exists a totally nondegenerate pg-diagonal 
of length $k$ with $ 0 \leq k \leq p-1$  whose complementary matrix 
has $\det^c$ nonzero.
If there exists such a diagonal with $k \geq 1$, say 
$\al_{i_1,j_1}, \ldots
, \al_{i_k, j_k}$, then also $ \al_{\hat{i_1},\hat{j_1}}$ does not satisfy (2), so, by induction assumption,   there exists $B \in 
\al_{\hat{i_1}, \hat{j_1}} $ with $\det(B) \neq 0$ and, as before, we can get a contradiction.
On the other hand, suppose that $\det^c(\al) \neq 0$ 
and the complementary matrix of every 
totally nondegenerate pg-diagonal 
of length $k$ with $ 1 \leq k \leq p-1$ 
has $\det^c$ equal to zero; we call this assumption ($\ast$).

Let $A \in \al$. 
By  (1), we can write $\det (A)$ as the sum of: 

- the  sum (with sign) 
of the product of the entries of the  pg-diagonals of $A$ of length $p$ such the corresponding entries of $\al$ are all degenerate, 

- the  sum (with sign) 
of the product of the entries of the  pg-diagonals of $A$  of length $p$ such all the corresponding entries of $\al$  apart from one are degenerate, 

....

- the  sum (with sign) 
of the product of the entries of the  pg-diagonals of $A$  of length $p$ such all the corresponding entries of $\al$  apart from $p-1$  are degenerate. 

We call ($\star$) this way to write $det(A)$.

The first sum coincides with  $\det^c(\al) $, so it is nonzero by the assumption ($\ast$); we can 
write the second sum by collecting the terms containing the same entry corresponding to the nondegenerate  entry of 
$\al$; so, by assumption ($\ast$), we get that 
this sum is  zero; we argue analogously for the other sums. So we can conclude that $\det(A)$ is nonzero, a contradiction.

$\Longleftarrow $ Let $\al$ be a  matrix 
satisfying (1) and (2) and let $A \in \al$. 
By  (1), we can write $\det (A)$ as in ($\star$).

The first sum is zero by assumption; we can 
write the second sum by collecting the terms containing the same entry corresponding to the nondegenerate  entry of 
$\al$; so by assumption we get that also 
this sum is  zero. We argue analogously for the other sums. 
\end{proof}

\begin{cor} \label{Mrk}
Let $\al $ be a subset matrix over a field $K$.
 Then $\Mrk(\al)$
is the maximum of the natural numbers $t$ such that there is a $ t \times t $ submatrix  of
 $\al$ either with a totally nondegenerate pg-diagonal  of  length  $t$ or with a totally nondegenerate
 pg-diagonal of  length between $0$ and $t-1$ whose complementary matrix has $\det^c  \neq 0$.
 \end{cor}

\bigskip

{\bf Acknowledgments.}
This work was supported by the National Group for Algebraic and Geometric Structures, and their  Applications (GNSAGA-INdAM).

The author wishes to thank the anonymous referee for his/her comments, which helped to improve the paper.

\end{document}